\documentclass[10pt,twoside,leqno]{amsart}
\usepackage{amssymb}
\theoremstyle{plain}
\newtheorem{thm}{Theorem}[section]
\newtheorem{cor}[thm]{Corollary}

\theoremstyle{definition}

\theoremstyle{remark}

\numberwithin{equation}{section}

\newcommand{\norm}[1]{\left\Vert#1\right\Vert}

\begin{document}
\setcounter{page}{1}


\vspace*{1.5cm}

\title[  ]
{\large The first common fixed point theorem for commutative set-valued mappings}
\author[I.Mohamadi]{Issa Mohamadi}
\address { \textit{Department of Mathematics,  Islamic Azad University - Sanandaj Branch,
  Sanandaj, Iran }}
\address {\textit{ E-mail addresses:
  imohamadi@iausdj.ac.ir\\imohamadi.maths@gmail.com}}

\date{}
\maketitle

\vspace*{-0.5cm}

\bigskip
\begin{abstract} We establish the first common fixed point theorem for commutative set-valued convex mappings. This may help to generalize
 common fixed point theorems in single-valued setting to those in set-valued. We also prove the existence of a fixed point in a continuously expanding sets under a none convex  upper semicontinuous set-vaued mapping; as a result we answer positively to a question of Lau and Yao. \end{abstract}
\maketitle
\bigskip



\section{Introduction }
 Let $X$ and $Y$ be two topological vector spaces, we recall that a set-valued mapping $T:X\rightarrow 2^{Y}$ is said to be upper semicontinuous, if for each open subset $V$ of $Y$ and each $x\in X$  with $T(x)\subseteq V$, there exists an open neighborhood $U$ of $x$ in $X$ such that $T(y)\subseteq V$ for all $y\in U$.
 For two set-valued mappings $T,S$ from $X$ into $2^{X}$,  their composition is defined, in the literature,  as $ToS(x)=\bigcup_{y\in S(x)}T(y)$ and $SoT(x)=\bigcup_{y\in T(x)}S(y )$.  $T$ and $S$ are also said to be commutative on $X$ if $ToS(x)=SoT(x)$, for all $x\in X$. We say that $T$ commutes with $ S$
on the right  if $SoT(x)\subseteq ToS(x)$, for all $x\in X$ . We say that a mapping $T$ from $X$ into $2^X$ is convex if $\lambda t+(1-\lambda)z\in T(\lambda x+(1-\lambda)y)$, for all $t\in T(x)$, $z\in T(y)$ and $\lambda\in (0,1)$. We also recall that for a set-valued mapping $T$ from $X$ into $2^{X}$, $x\in X$ is a fixed point of $T$ if $x\in T(x)$.

 Let $(X,d)$ be a metric space and $CB(X)$ denote the set of nonempty closed bounded subset of $X$. For $A,B\in CB(X)$, define
 $$H(A,B)=\max \{\delta(A,B),\delta (B,A)\} $$
  where, $\delta(A,B)=\sup\{d(a,B): a\in A\}$ and $\delta (B,A)=\sup\{d(A,b): b\in B\}$.
 It is known that $(CB(X),H)$ is a metric space. The metric $H$ on $CB(X)$ is called the Hausdorff metric.\par 
 A mapping $T$ from a metric space $(x,d)$ into the metric space $(CB(X),H)$ is said to be nonexpansive if $H(T(x),T(y))\leq d(x,y)$, for all $x,y\in X$. 
 
 Suppose that $C$ is a nonempty subset of a topological space $X$ and $D$ is a nonempty subset of $C$. The mapping $R:C\longrightarrow D$ is said to be a retraction if $R(x)=x$ for all $x\in D$; that is, $R^{2}=R$. In this case, $D$ is called a retract of $C$. When $(X,d)$ is a metric space then $D$ is called a nonexpansive retract of $C$ if $R$ is a nonexpansive mapping.

   For more details on these and related concepts refere to \cite{Aubin}. 
 
 There are a number of landmark fixed point theorems for set-valued mappings. In $1941$, Kakutani \cite{Kakutani} showed that if $C$ is a nonempty convex compact subset of an $n$-dimentional Euclidean space $\mathbb{R}^{n}$ and $T$ from $C$ into $2^{C}$ is an upper semicontinuous mapping such that $T(x)$ is a nonempty convex closed subset of $C$ for all $x\in C$; then, $T$ possesses a fixed point in $C$. In $1951$, Glicksberg \cite{Glicksberg} and in $1952$, Fan \cite{Fan}, independently, generalized Kakutani's fixed point theorem $[5]$ from Euclidean spaces to locally convex vector spaces. In \cite{Issa}, we showed that for a continuously expanding compact and convex subset of a locally convex vector space, under an upper semicontinuous set-valued convex mapping, there exists at least one point that remains fixed under the expansion. In this work we generalize this result to an arbitrary upper semicontinuous set-valued mapping in one dimensional Euclidean space $\mathbb{R}$. \par 
 Many common fixed point theorems for single-valued mappings have also been developed; among them, the Markov-Kakutani fixed point theorem is of great interest for its numerous verity of applications that can be found in the literature. In $1936$, Markov \cite{Markov} and in $1938$, Kakutani \cite{Kakutani1}  proved, independently, that each family of commutative continuous affine mappings on a nonempty compact convex subset of a Hausdorff topological vector space into itself has a common fixed point. A part of our work has also been devoted to generalize their theorem, applying our fixed point theorem along with the Fan-Glicksberg fixed point theorem, for a family of two but convex and set-valued mappings. The last part of our work is also  devoted to provide an answer to a question by Lau and Yao \cite{Lau}. In fact, we generalize our common fixed point theorem for none convex set-valued mappings in one dimensional Euclidean space.
 .\par  \section{Our results}

In the following theorem, we prove the existence of a common fixed point for two set-valued convex mappings.

\begin{thm}  Let $X$ be a locally convex Hausdorff vector space, and $C$ be a nonempty convex compact subset of $X$. Suppose that $\{T_{1}, T_{2}\}$ are two commutative upper semicontinuous convex set-valued mappings from $C$ into $2^{C}$ such that $T_{i}(x)$, for $i=1,2$ and $x\in C$, is a nonempty closed subset of $X$. Then, there exists $x\in C$ such that $x\in T_{1}(x)\cap T_{2}(x)$.
\end{thm}
\begin{proof}
Let $Fix(T_{i})$ indicates the fixed points set of $T_{i}$, for $i=1,2$. Then, by the Fan-Glicksberg fixed point theorem, $Fix(T_{1})$ is nonempty compact convex subset of $X$. Define $G:Fix(T_{1})\rightarrow 2^{Fix(T_{1})}$ by $G(x)=T_{2}(x)\cap Fix(T_{1})$, for $x\in Fix(T_{1})$. Then, $G$ is an upper semicontinuous set-valued mapping in the topology on $Fix(T_{1})$ induced from $X$. Now, we show that $G(x)$ is nonempty. Let $x\in Fix(T_{1})$, then $x\in T_{1}(x)$. Accordingly, $T_{2}(x)\subseteq T_{2}(T_{1}(x))=T_{1}(T_{2}(x))$ by commutativity of $T_{1}$ and $T_{2}$ and definition of composition for set-valued mappings. Since $T_{1}(x)$ is a nonempty convex compact subset of $X$, by Theorem $2.2$ in \cite{Issa}, $T_{1}$ has a fixed point in $T_{2}(x)$. That is , there exists $y\in T_{2}(x)$ such that $y\in T_{1}(y)$. It yields that $G(x)$ is nonempty. Therefore, by the Fan-Glicksberg fixed point theorem, again, $G$ has a fixed point on $Fix(T_{1})$. Thus, there exists $x\in Fix(T_{1})$ so that $x\in G(x)$. This means $x\in T_{2}(x)\cap Fix(T_{1})$. This completes the proof.
\end{proof}

{\bf Open problem 1.} We still don't know whether or not Theorem $2.1$ is valid for a family of infinite number of  commutative set-valued convex mappings; that is, whether or not a generalization of the Markov-Kakutani fixed point theorem to commutative set-valued convex mappings holds.\par  
{\bf Remark.} In Theorem $2.1$ instead of commutativity we can suppose that $T_{1}$ commutes with $T_{2}$ on the right.\\

Next, we generalize Theorem $2.2$ in \cite{Issa} for an arbitrary upper semicontinuous set-valued mapping in one dimesional Euclidean spaces.
\begin{thm}
Let $C$ be a nonempty convex compact subset of $\mathbb{R}$. Assume that $T:C\rightarrow 2^{\mathbb{R}}$ is a set-valued upper semicontinuoues mapping such that $T(x)$ is a nonempty compact convex subset of $\mathbb{R}$ for all $x\in C$. If $C\subseteq T(C)$, then $T$ possesses a fixed point in $C$.
\end{thm}
\begin{proof}
 Let $$\Delta =\{U\subseteq C: U is\:nonempty,\:closed,\:convex\:and\:U\subseteq T(U)\}.$$
Then, $(\Delta,\subseteq)$, where $\subseteq$ is inclusion, is a partially ordered set. Also, by Lemma $2.1$ in \cite{Issa}, every descending chain in $\Delta$ has a lower bound in $\Delta$. Therefore, by Zorn's lemma, $\Delta$ has a minimal element, say $U_{0}$. We show that $U_{0}$ is singleton. Define $F:U_{0}\rightarrow 2^{U_{0}}$ by $F(x)=T(x)\cap U_{0}$ for all $x\in U_{0}$. Then, $F(x)$ is a convex, compact subset of $X$, for all $x\in U_{0}$ since $T(x)$ and $U_{0}$ are convex and compact.  Let $V=\{y\in U_{0} : S(x)\neq\emptyset\}$. Then, $V$ is nonempty as $U_{0}\subseteq T(U_{0})$. It is clear that $V\subseteq T(V)$. By convexity and upper semicontinuity of $T$, it can easily be seen that $V$ is a nonempty compact subset of $U_{0}$ such that $V\subseteq T(V)$ and $V\neq U_{0}$.Then, $U_{0}=co(V)$ by minimality of $U_{0}$ and the fact trhat $V$ is a compact subset in $\mathbb{R}$. Also, since $U_{0}$ is a nonempty convex compact subset of $\mathbb{R}$, it is a closed interval, say $[a,b]$, where $a,b\in V$. In fact,We shall prove that $a= b$, We show it by the way of contradiction; that is, we suppose that $a\neq b$.  Now, let
 $$\Omega=\{[c,d]\subseteq [a,b]: T(c)\cap [c,d]\neq\emptyset\: and  \:\: T(d)\cap [c,d]\neq\emptyset\}.$$ 
Then, $\Omega\neq\emptyset$ Since $[a,b]\in \Omega$. We show that $\Omega$ has a minimal element. Let $\{[c_{i},d_{i}]\}_{i\in I}$ be a descending chain, by inclusion, in $\Omega $. Thus, $\bigcap_{i\in I}[c_{i},d_{i}]$ is a nonempty compact convex subset in $\mathbb{R}$, and therefore a closed interval, say $[c,d]$. By defining the relation $\leq$ on $I$ as : $i\leq j$ iff $[c_{j},d_{j}] \subseteq [c_{i},d_{i}]$, for all $i,j\in I$, $(I,\leq)$ is a directed set. Accordingly, $c=\lim_{i}c_{i}$ and $d=\lim_{i}d_{i}$. Next, we show that $T(c)\cap [c,d]\neq\emptyset\: and  \:\: T(d)\cap [c,d]\neq\emptyset$. Suppose, on contrary, that $T(d)\cap [c,d]=\emptyset$ Therefore, by upper semicontinuity of $T$ there exists an open neibourhood $U$ and $W$ containing $[c,d]$ and $T(d)$ , respectively, such that $ U\cap W=\emptyset$. Also, there exists a neiborhood $U'$ of $d$ so that for all $x$ in $U'$ we have $T(x)\subseteq W$. This implies that there is $i_{0}\in I$ such that $T(d_{i})\subseteq W$ for all $i\geq i_{0}$. On the other hand, there also exists $i_{1}\in I$ that $c_{i},d_{i}\in U$, for all $i\geq i_{1}$. Now having taken $i\geq max\{i_{0},i_{1}\}$ it follows that $[c_{i},d_{i}]\cap T(d_{i})=\emptyset$, a cotradiction. Accordingly, $[c,d]\in \Omega$. Thus, by Zorn's lemma $\Omega$ has a minimal element, say $[c',d']$. If $T(x)\cap [c',d']\neq\emptyset$, for all $x\in [c',d']$; then  by defining $P:[c',d']\longrightarrow 2^{[c',d']}$ as $P(x)=T(x)\cap[c',d']$ and also applying Kakutani's fixed point theorem for mapping $P$, it follows that $T$ has a fixed point in $[c',d']$. This contradicts the minimality of $U_{0}$. Therefore, $T(y)\cap [c',d']=\emptyset$, for some $y\in [c',d']$. Now, suppose that $T(y)>d'$ (to avoid any incovenience, by $T(y)>d'$ we mean $z>d'$ for all $z\in T(y)$) and define
$$\Theta=\{U\subseteq [c',d']: U\: is \: an\: open \: interval \: containing \: y\: and\: T(w)>d'\: for\: all\: w\in U\}.$$

By upper semicontinuity of $T$, $\Theta$ is nonempty. Also , by applying Zorn's lamma, $\Theta$ has a maximal element, by inclusion, such as $U=(s,t)$. Hence, upper semicontinuity of $T$ also implies that $T(t)\cap [c',d']\neq\emptyset$. We shall prove that $d'\in T(t)$. Let $\{x_{n}\}$ and $\{y_{n}\}$ be sequences such that $x_{n}\rightarrow t^{-}$; and $y_{n}\in T(x_{n})$. Thus,  $y_{n}>d'$. Since $T$ is upper semicontinuous and compact valued and $C$ is compact, it is known that $T(C)$ is also compact. Accordingly, by passing to a subsequence we may assume that $y_{n}\rightarrow y$, for some $y\in \mathbb{R}$. Hence, $y\in T(t)$ and $y\geq d'$. It follows that $d'\in T(t)$ as $T(t)$ is a nonempty compact convex subset of $\mathbb{R}$. If $T(d')\cap [t,d']\neq\emptyset$, it contradicts the minimality of $[c',d']$. Accordingly, we may suppose that $T(d')<t$ as $T(d')\cap [c',d']\neq\emptyset$. Now, let
$$\Sigma=\{[m,n]\subseteq [t,d']: either\:n\in T(m), T(n)<m,\: or \: m\in T(n), T(m)>n \}.$$ 

Then, $\Sigma$ is a nonempty set since $[t,d']\in \Sigma$. Let $\{[m_{j},n_{j}]\}_{j\in J}$ be a descending chain, by inclusion, in $\Omega $; then, by the same argument we had for $[c',d']$, for $[m,n]=\bigcap_{j\in J}[m_{j},n_{j}]$ we have $n\in T(m)$ or $m\in T(n)$. We shall prove that $[m,n]\in \Sigma$; for this , having assumed that $n\in T(m)$  it is enough to show that $T(n)<m$. Suppose, on contrary, that there exists $x\in T(n)$ so that $x\geq m$. If $T(n)\cap [m,n]\neq\emptyset$, then it contradicts the minimality of $[c',d']$. Thus, we may assume that $T(n)>n$. Let $O_{1}$ and $O_{2}$ be open sets containing $[m,n]$ and $T(n)$, respectively, and also $O_{1}\cap O_{2}=\emptyset$. Then by upper semicontinuity of $T$ it follows that there exists $j_{0}\in J$ such that $m_{j}\in O_{1}, n_{j}\in O_{1}$,  and $T(n_{j})\in O_{2}$, for all $j\geq j_{0}$. Accordingly,   $T(n_{j})>n_{j}$, for all $j\geq j_{0}$. This contradicts the fact that $m_{j}\in T(n_{j})$, for all $j\geq J_{0}$. Hence, by Zorn's lemma, $(\Sigma,\subseteq)$ has a minimal element such as $[m',n']$. Suppose that $m'\in T(n')$ and $T(m')>n'$, then by the same disscusion we already had for $[t,d']$, there exist $p'$ in $(m',n']$ so that $n'\in T(p')$. This contradicts either the minimality of $[m',n']$ or the minimality of $[c',d']$.\par 
 The similar argument can be repeated with minor alterations for the case when we have $T(y)<c'$. Therefore, any case yields a contradiction. Thus, $a=b$; that is $U_{0}$ is singleton. This complete the proof.   

 \end{proof}

Also, from the proof of Theorem $2.4$, the following result can be derived.
\begin{cor} Let $[a,b]$ be a closed interval in $\mathbb{R}$, and $T:[a,b]\rightarrow 2^{\mathbb{R}}$ be a set-valued upper semicontinuoues mapping such that $T(x)$ is a nonempty compact convex subset of $\mathbb{R}$ for all $x\in [a,b]$. Suppose also that $T(a)\cap [a,b]\neq\emptyset$ and $T(b)\cap [a,b]\neq\emptyset$. Then, $T$ possesses a fixed point in $[a,b]$.
\end{cor} 

The following example shows that Theorem $2.2$ is not valid in more general spaces.\par 
{\bf Example.} Let $T$ be the set-valued mapping from $C=[0,2]$ into ${2^{\mathbb{R}}}^{2}$ defined by
$$T(x)=\begin{cases}[1,2-x]\times \{x\}; & x\in [0,1],\\ [2-x,1]\times \{2-x\};& x\in (1,2],\end{cases}.$$   
 where $\times $ is the Cartesian product. It is obvious that $C\subset T(C)$ as we have $T(0)=[1,2]$ and $T(2)=[0,1]$. It can easily be verified that $T$ is a nonempty convex compact upper semicontinuous  set-valued mapping that does not possess any fixed point in $C$.\par 

This example gives rise to the following open problem:\par
{\bf Open problem 2.} As it can be seen from the above example, Codim$(\frac{\mathcal {M}(T(C))}{\mathcal {M}(C)})\neq 0$  but in Theorem $2.2$ it is zero. Now the question that whether Theorem  $2.2$  holds in more general spaces where we have it zero, is still unanswerd, where by $\mathcal {M}(T(C))$ and $\mathcal {M}(C)$ we mean   the subspacees of $X$ containing $T(C)$ and $C$, respectively, with minimum dimensions.

 In what follows we prove the existence of a  common fixed point for a family of  commutative none convex set-valued mappings. Not only it provides an answer to question $5.9$ in \cite{Lau} but also it gives an insight into the structure of the set of common fixed points for set-valued mappings.
 
 \begin{thm}  Let  $C$ be a nonempty convex compact subset of $\mathbb{R}$. Suppose that $\Psi=\{T_{i}: i\in I\}$ is a family of  commutative nonexpansive set-valued mappings from $C$ into $2^{C}$ in which there are  at most two mappings that are not singled valued. If for each $i\in I$ and $x\in C$, $T_{i}(x)$ is a nonempty closed convex subset of $C$, then the common fixed points of $\Psi$ is a nonempty convex  nonexapansive retract of $C$.
\end{thm}
\begin{proof}  For  $i\in I$ we show that $Fix(T_{i})$ is convex. For each $x\in C$, define 
$$f_{i}(x)=P_{T_{i}(x)}(x)=\{y\in T_{i}(x): d(x,y)=\inf \{d(x,z):z\in T_{i}(x)\}\},$$ where $ P_{T_{i}(x)}$ is the metric projection on $T_{i}(x)$ for each $x\in C$. It can easily be seen that  $Fix(f_{i})=Fix(T_{i})$. To avoid any complexity in writing,  by $x\leq T_{i}(y)$ we mean $x\leq z$ for all $z\in T_{i}(y)$ and by $T_{i}(x)\leq T_{i}(y)$ we mean $w\leq z$ for all $w\in T_{i}(x)$ and $z\in T_{i}(y)$  . We shall prove that $f_{i}$ is a nonexpansive mapping from $C$ into $C$. Since $T_{i}$ is a nonempty closed convex  mapping in $\mathbb{R}$, we may suppose  $T_{i}(x)=[a,b], T_{i}(y)=[c,d]$ for $x,y\in C$. We consider the following cases:\par 
Case $1$. Either $x\leq y\leq T_{i}(x)\leq T_{i}(y)$ or $x\leq T_{i}(x)\leq y\leq  T_{i}(y)$; then by definition of $f_{i}$ we have $f_{i}(x)= a, f_{i}(y)=c$, therefore, 
$$\norm{ f_{i}(x)- f_{i}(y)}\leq H(T_{i}(x),T_{i}(y))\leq \norm {x-y}.$$

Case 2. In either case $x\leq   T_{i}(y)\leq y\leq T_{i}(x) $ or $x\leq   T_{i}(x)\leq T_{i}(y) \leq y$ we have  $f_{i}(x)= a, f_{i}(y)=d$ and it is clear that $\norm{ f_{i}(x)- f_{i}(y)}\leq \norm {x-y}$.\par 

Case 3. By nonexpansivity of $H$ the case when we have $   T_{i}(x)< x, y< T_{i}(y) $ is also imposssible.\par 

We can consider other cases by replacing $\leqslant$ by $\geqslant$ in the mentioned cases and obtain the same result. Next we show that $Fix(f_{i})$ is convex. Let $x,y\in Fix(f_{i})$ and $0\leq \lambda\leq 1$, then for $z=\lambda x+(1-\lambda)y$ we have 
$$\norm {x-f_{i}(z)}=\norm {f_{i}(x)-f_{i}(z)}\leq\norm{x-z}=(1-\lambda)\norm{x-y},$$
$$\norm {y-f_{i}(z)}=\norm {f_{i}(y)-f_{i}(z)}\leq\norm{y-z}=\lambda\norm{x-y}.$$
These yields
$$\norm {x-y}\leq \norm{x-f_{i}(z)}+\norm{f_{i}(z)-y}\leq\norm{x-z}+\norm {y-z}=\norm{x-y}.$$
Consequently, $\norm{x-y}=\norm {x-f_{i}(z)}+\norm{f_{i}(z)-y}$. We show that $x\leq f_{i}(z)\leq y$. Suppose, on contrary, that either $ f_{i}(z)\leq x$ or $y\leq f_{i}(z)$; each case results in either $\norm{ f_{i}(z)-y}>\norm{x-y}$ or  $\norm{ f_{i}(z)-x}>\norm{x-y}$, respectively; which is a contradiction. Hence, there is $\mu \in [0,1]$ such that $f_{i}(z)=\mu x+(1-\mu) y$. Thus, 
$$\mu\norm {x-y}= \norm {y-f_{i}(z)}=\norm {f_{i}(y)-f_{i}(z)}\leq\norm{y-z}=\lambda\norm{x-y},$$
$$(1-\mu)\norm {x-y}= \norm {x-f_{i}(z)}=\norm {f_{i}(x)-f_{i}(z)}\leq\norm{x-z}=(1-\lambda)\norm{x-y}.$$
Accordingly, $\lambda  = \mu$, that is, $f_{i}(z)=z$. This proves that  $Fix(T_{i})=Fix(f_{i})$ is convex.\par 

 By finite intersection property for compact sets we may suppose that $I=\{1,2,...,n\}$ where $n\in \mathbb{N}$. Let $F_{n}=\cap_{i=1}^{n}  Fix( T_{i})$.\par 
 The proof is by induction. For $n=2$,  assume that neither $T_{1}$ nor $T_{2}$ are single valued. Following the proof of Theorem $2.1$ and applying Theorem $2.2$, it follows that $F_{n}$ is a nonempty  convex subset of C .\par 
We shall prove that $F_{2}$ is a nonexpansive retract of $C$. It is known, by Bruck \cite{Bruck}, that for the nonexpansive single valued mapping $f_{1}$, already defined, there exists a nonexpansive retraction $g_{1}$ from $C$ onto $Fix(f_{1})=Fix(T_{1})=F_{1}$. Now define $S:C\rightarrow 2^{C}$ by
  $$S(x)=T_{2}(g_{1}(x))\cap Fix(T_{1}).$$
  Then, it is easy to verify that 
  $$H(S(x),S(y))\leq H(T_{2}(g_{1}(x)),T_{2}(g_{1}(y)))\leq \norm{g_{1}(x)-g_{1}(y)}\leq \norm{x-y}.$$
  
 Having noticed $g_{1}(x)=x$ and following the proof of Theorem $2.1$, it yields that $S(x) $ is a nonempty convex compact subset of $C$ for all $x\in C$.
  On the other hand, for $x\in Fix(S)$ we have $x\in Fix(T_{1})$; thus  $g_{1}(x)=x.$ Therefore, $x\in T_{2}(x)$; that is,  $Fix(S)\subseteq F_{2} $. The inclusion $F_{2} \subseteq Fix(S)$ is also clear. Accordingly, the first part of the proof implies that $F$ is a nonexpansive retract of $C$.  \par 
  Now, let $n\geq 3$,  $F_{n-1}\neq\emptyset$ and $r:C\rightarrow F_{n-1}$ be its correspondant retraction. Then, we show that Fix$(T_{n}\texttt{o} r)=F_{n}$. The inclusion $F_{n}\subseteq Fix(T_{n}\texttt{o} r)$ is trivial. For the reverse inclusion, let $x\in Fix (T_{n}\texttt{o} r)$. Since $T_{n}$ commutes with $T_{i}$ for $i=1,...,n-1$ and $r(x)\in F_{n-1}$, $F_{n-1}$ is $T_{n}$ invariant and $x=T_{n}\texttt{o} r(x)\in F_{n}$. Therefore, $ r(x)=x$. That is, $x=T_{n}\texttt{o} r(x)=T_{n}(x)$. Accordingly,  Fix$(T_{n}\texttt{o} r)\subseteq F_{n}$. Applying Bruck's theorem for the nonexpansive mapping $T_{n}\texttt{o} r$, the proof is completed.

\end{proof}
         
{\bf Remark.} In \cite{Boyce}, Boyce gave an example of two commutative mappings that have no common fixed point. This shows that the condition that the mappings in Theorem $2.4$ are nonexpansive can not be dropped.


\end{document}